\documentclass[10pt, a4paper]{article}
\usepackage{color,url,latexsym,amssymb,amsthm,amsmath,amsxtra,amsfonts}
\usepackage{amsfonts}
\usepackage{amsbsy}
\usepackage{amssymb}
\usepackage{amsmath}
\usepackage{amsmath,amscd}
\usepackage{amsthm}
\usepackage{fancyhdr}

\theoremstyle{plain}
\newtheorem{theorem}{Theorem}[section]

\newtheorem{corollary}{Corollary}[section]
\newtheorem{proposition}{Proposition}[section]
\theoremstyle{definition}
\newtheorem{definition}{Definition}[section]
\newtheorem{remark}{Remark}[section]
\newtheorem{example}{Example}[section]

\begin{document}

\title{Holomorphic last multipliers on complex manifolds}
\author{Mircea Crasmareanu, Cristian Ida and Paul Popescu}
\date{}
\maketitle
\begin{abstract}
The  goal of this paper is to study the theory of last multipliers in the framework of complex manifolds with a fixed holomorphic volume form. The motivation of our study is based on the equivalence between a holomorphic ODE system and an associated real ODE system and we are interested how we can relate holomorphic last multipliers with real last multipliers. Also, we consider some applications of our study for holomorphic gradient vector fields  on holomorphic Riemannain manifolds as well as for holomorphic Hamiltonian vector fields and holomorphic Poisson bivector fields on holomorphic Poisson manifolds.
\end{abstract}

\medskip
\begin{flushleft}
\strut \textbf{2010 Mathematics Subject Classification:} 34A26, 58A15, 34C40, 53C56, 53D17.

\end{flushleft}

\section{Introduction and Preliminaries}
\setcounter{equation}{0}
\subsection{Introduction}

The last multipliers are very useful tools in the study of completely integrate systems of first-order ODE's \cite{Fl}. More precisely, if the given system $\Sigma $ has $n$ equations and we already know $(n-2)$ first integrals of $\Sigma $ then, by using last multipliers, we can obtain a new first integral and hence we can approach the topics of complete integrability or super-integrability. Also, the last multipliers  sometimes allow the construction of an associated Lagrangian in the case of systems of second-order ODE's \cite{Nu-Le-JMP}, so they are useful for the inverse problems, and for integrating nonholonomic systems \cite{Z}.

Their history begins with Jacobi as it is pointed out in \cite{C2} and hence, sometimes they appear in literature as {\it Jacobi last multipliers} as in \cite{Nu, Nu-Le, Nu-Le-JMP} and references therein. For all those interested in historical aspects we refer to the survey \cite{B-G}. Until recently, their use was restricted to systems on Euclidean spaces \cite{Oz}. The first named author of this work initiated their study on manifolds in \cite{C1} and \cite{C2}, where he pointed out their relationship with the Liouville equation of transport. Since then, the last multipliers  have been considered in various (non-flat) settings such as: Riemannian and Poisson geometry in \cite{C2}, Lie-Poisson structures in \cite{C3}, weighted manifolds in \cite{C4}, Lie algebroids in \cite{c:h}.

The present paper makes a new extension, namely one  referring to the complex geometry framework. More precisely, we start with a complex manifold and work in the category of holomorphic objects. So, the considered volume forms, vector fields, functions as well as Riemannian or Poisson structures are all supposed to be holomorphic. The motivation of our study is based on the equivalence between a holomorphic ODE system and an associated real ODE system. In this way, we obtain that a holomorphic last multiplier for a holomorphic vector field defines a real last multiplier for two associated real vector fields and conversely (Theorem \ref{thlm}). 

The contents of the paper is as follows: in the next subsection we fix a complex manifold of dimension $n$ and a holomorphic volume form $\omega $ and we recall some properties of the divergence of a holomorphic vector field with respect to $\omega$ as basis of our work. The following section describes the general approach of holomorphic last multipliers for holomorphic vector fields with respect to $\omega $ as well as their general properties. When the complex manifold is endowed with a holomorphic Riemannian metric, we discuss the case of gradient vector field of a holomorphic function and we relate its holomorphic last multipliers with real last multipliers for gradient vector fields associated with anti-K\"{a}hlerian metrics defined by real and imaginary parts of the holomorphic Riemannian metric (Theorem \ref{tg}). Also, some methods for obtaining inverse holomorphic multipliers on complex manifolds will be considered. The case of holomorphic Poisson manifolds is studied separately in the third section, where the particular class of holomorphic Hamiltonian vector fields associated with holomorphic Poisson structures with a special view towards the unimodular case and some examples are considered. Also, we relate the holomorphic last multipliers for holomorphic Hamiltonian vector fields on holomorphic Poisson manifolds with real last multipliers for Hamiltonian vector fields corresponding to some natural real Poisson structures on the underlying real manifold (Theorems \ref{th1} and \ref{th2}). In the subsection 3.3 we extend holomorphic last multipliers from holomorphic vector fields to holomorphic multivectors and we study carefully the Poisson bivector as a remarkable example. Concerning practical examples we consider some particular cases in  low dimensions $n=2$ and  $n=3$, respectively.

\subsection{Preliminaries}

Let $M$ be a $n$-dimensional complex manifold. A holomorphic section of its holomorphic tangent bundle $T^{1,0}M$ defines a holomorphic vector field on $M$, and we denote by $\mathcal{X}_{\mathcal{O}}(M)$ the Lie algebra of holomorphic vector fields on $M$. Denote by $\Omega^p_{\mathcal{O}}(M)$ the set of holomorphic $p$-forms on $M$, and  suppose that $M$ admits a nowhere vanishing holomorphic $n$-form $\omega$, which is also called a \textit{holomorphic volume form} on $M$. There are many examples of complex manifolds which posses holomorphic volume forms.
\begin{enumerate}
\item[(i)] Let $(M,\Omega)$ be a $2n$-dimensional holomorphic symplectic manifold, i.e. $\Omega$ is a  nondegenerate closed holomorphic $2$-form on $M$. Then $\omega:=\Omega^n$ is a holomorphic volume form on $M$.
\item[(ii)] Let $G$ be a $n$-dimensional complex Lie group, and $\{Z_k\}$, $k=1,\ldots,n$ be  any basis of left invariant holomorphic vector fields of $\mathcal{X}_\mathcal{O}(G)$. If $\{\theta^k\}$, $k=1,\ldots,n$ are their dual left invariant holomorphic one forms, that is, $\theta^k(Z_j)=\delta^k_j$, then $\omega:=\theta^1\wedge\ldots\wedge\theta^n$ defines a holomorphic volume form on $G$.
\item[(iii)] A holomorphic Riemannian metric $g$ on $M$ is defined as a global holomorphic section of the symmetric holomorphic forms $\odot^2 (T^{1,0}M)^*$ such that the "index-lowering" map $T^{1,0}M\rightarrow (T^{1,0}M)^*$ (obtained by indentifying $\otimes^2(T^{1,0}M)^*$ with ${\rm Hom}(T^{1,0}M,(T^{1,0}M)^*)$ in the tautological fashion) is an isomorphism, see \cite{LeB}. In local complex coordinates $(z^1,\ldots,z^n)$ on $M$, a holomorphic metric $g$ appears as $g=\sum\limits_{j,k}g_{jk}dz^jdz^k$, where $\det(g_{jk})\neq0$, and for every $j,k,l=1,\ldots,n$ we have $\partial g_{jk}/\partial \overline{z}^l=0$ and $g_{jk}=g_{kj}$. For instance, $g=\sum\limits_{i=1}^n(dz^i)^2$ defines a holomorphic Riemannian metric on $\mathbb{C}^n$, called the \textit{holomorphic euclidean} metric. Also, if $G$ is a $n$-dimensional complex Lie group, and $\{Z_a\}$, $a=1,\ldots,n$ is a basis of its holomorphic Lie algebra, then using local complex coordinates $(z^1,\ldots,z^n)$ on $G$,  we can write $Z_a=\chi^j_a(z)(\partial/\partial z^j)$ where ${\rm rank}(\chi^j_a)_{n\times n}=n$  and $\chi^j_a(z)$ are holomorphic functions on $G$. Then, a holomorphic Riemannian metric on $G$ can be defined by seting $g_{jk}=\delta_{ab}\chi^a_j\chi^b_k$, where $(\chi^a_j(z))=(\chi^j_a(z))^{-1}$. Moreover, if $G$ is semi-simple and $C^c_{ab}$ are complex constant of structure of $G$, that is $[Z_a,Z_b]=C^c_{ab}Z_c$, then the matrix of complex Cartan-Killing  elements $C_{ab} = C^{c}_{ad}C^{d}_{cb}$ of $G$ is invertible, and a new holomorphic Riemannian metric on $G$ can be defined by setting $g_{jk}=C_{ab}\chi^a_j\chi^b_k$.

Such a holomorphic Riemannian metric $g$ defines a \textit{holomorphic metric volume form} $\omega_g$, see \cite{LeB}, as a global holomorphic $n$-form on $M$ such that
\begin{displaymath}
\omega_g(E_1,\ldots,E_n)=\pm1,
\end{displaymath}
where $\{E_1,\ldots,E_n\}$ is an orthonormal holomorphic frame on $(M,g)$, that is $g(E_j,E_k)=\delta_{jk}$, $j,k=1,\ldots,n$. If $(M,g)$ admits such a volume element, it admits precisely two of them.

\end{enumerate}

According to \cite{T-V}, such a form can be used in the definition of the holomorphic divergence with respect to $\omega$, that is a map ${\rm div}_{\omega}:\mathcal{X}_{\mathcal{O}}(M)\rightarrow\mathcal{O}(M)$ given by
\begin{equation}
\label{I1}
{\rm div}_{\omega}(Z)\omega=\mathcal{L}_Z\omega
\end{equation}
where $\mathcal{L}_Z$ is the Lie derivative with respect to $Z\in\mathcal{X}_\mathcal{O}(M)$.

For instance if $\{Z_k\}$, $k=1,\ldots,n$ is any basis of left invariant holomorphic vector fields of a $n$-dimensional complex Lie group $G$, and if $Z=\sum\limits_{k=1}^nZ^kZ_k$ is any holomorphic vector field on $G$, then ${\rm div}_{\omega}(Z)=\sum\limits_{k=1}^n Z_k(Z^k)$.

Taking into account $\mathcal{L}_{[Z,W]}=[\mathcal{L}_Z,\mathcal{L}_W]$ it follows that
\begin{equation}
\label{I2}
{\rm div}_{\omega}([Z,W])=Z({\rm div}_\omega (W))-W({\rm div}_\omega (Z)),\,\forall\,Z,W\in\mathcal{X}_\mathcal{O}(M).
\end{equation}
Moreover, using the Cartan's formula $\mathcal{L}_Z=\partial\circ \imath_Z+\imath_Z\circ \partial$ (where $d=\partial+\overline{\partial}$ is the usual decomposition of the exterior derivative and $\imath_Z$ is the interior product with respect to a holomorphic vector field $Z$), we obtain
\begin{equation}
\label{I3}
{\rm div}_\omega(Z)\omega=\partial(\imath_Z\omega),
\end{equation}
and, it is easy to see that
\begin{equation}
\label{I4}
{\rm div}_\omega(f\cdot Z)=Zf+f{\rm div}_\omega(Z)
\end{equation}
for every holomorphic function $f\in\mathcal{O}(M)$ and every holomorphic vector field $Z\in\mathcal{X}_\mathcal{O}(M)$. Also, for our next considerations it is natural to consider the set of holomorphic first integrals of a holomorphic vector field $Z$, that is $I^1_\mathcal{O}(Z)=\{f\in\mathcal{O}(M)\,|\,Zf=0\}$, see \cite{T-V}.

\section{Holomorphic last multipliers for holomorphic vector fields}
\setcounter{equation}{0}
In this section we describe the general approach of holomorphic last multipliers for holomorphic vector fields with respect to a holomorphic volume form as well as their general properties. Also, we consider the case of gradient vector fields on holomorphic Riemannian  manifolds and some methods to obtain inverse holomorphic multipliers on complex manifolds.

\subsection{Basic definitions and results on holomorphic last multipliers}
Assume that the complex manifold $M$ is endowed with a holomorphic volume form $\omega\in\Omega^n_\mathcal{O}(M)$.  
Let $Z=\sum\limits_{i=1}^nZ^i(z)(\partial/\partial z^i)\in\mathcal{X}_{\mathcal{O}}(M)$ be a holomorphic vector field on $M$ written in local complex coordinates $(z^1,\ldots,z^n)$ on $M$, $\theta=\imath_Z\omega\in\Omega^{n-1}_{\mathcal{O}}(M)$ and
\begin{equation}
\label{comsys}
\frac{d z^k}{dt}=Z^k(z^1(t),\ldots,z^n(t)),\,1\leq k\leq n\,,\,t\in\mathbb{R}
\end{equation}
a complex ODE system on $M$ defined by the holomorphic vector field $Z$.

Then, in relation with the classical definition of a last multiplier function for a vector field on smooth manifolds, we consider  the following.

\begin{definition}
\label{d1}
A holomorphic function $\alpha\in\mathcal{O}(M)$ is called a \textit{holomorphic last multiplier} of the complex ODE system generated by $Z$ (or holomorphic last multiplier for $Z\in\mathcal{X}_\mathcal{O}(M)$) if
\begin{equation}
\label{II1}
\partial(\alpha \theta):=\partial\alpha\wedge\theta+\alpha\cdot \partial\theta=0.
\end{equation}
\end{definition}
The above definition of holomorphic last multipliers on complex manifolds  has the following characterizations in terms of some cohomological operators.

If the $(n-1)$-th holomorphic de Rham cohomology group of $M$ vanishes, then $\alpha\in\mathcal{O}(M)$ is a holomorphic last multiplier for $Z\in\mathcal{X}_{\mathcal{O}}(M)$ if and only if there is $\eta\in \Omega^{n-2}_\mathcal{O}(M)$ such that $\alpha\theta=\partial\eta$.

Also, let us consider the holomorphic version of Marsden differential, which is defined as  follows: for a holomorphic function $f\in\mathcal{O}(M)$ we consider the operator $\partial^f:\Omega^\bullet_\mathcal{O}(M)\rightarrow\Omega^{\bullet+1}_{\mathcal{O}}(M)$ given by $\partial^f\varphi=(1/f)\partial(f\varphi)$. Then $\alpha\in\mathcal{O}(M)$ is a holomorphic last multiplier for $Z\in\mathcal{X}_\mathcal{O}(M)$ if and only if $\theta$ is $\partial^\alpha$-closed.

Moreover, another characterization of the holomorphic last multipliers on complex manifolds can be given using the cohomology attached to a function introduced  in \cite{Mo1,Mo2} as  follows: if $f\in\mathcal{O}(M)$ and $k\in\mathbb{Z}$,  then we can define the linear operator $\partial^{(k)}_f:\Omega_\mathcal{O}^p(M)\rightarrow\Omega_\mathcal{O}^{p+1}(M)$ by
\begin{equation}
\label{a1}
\partial^{(k)}_f\varphi=f\partial\varphi-(p-k)\partial f\wedge\varphi\,,\,\forall\,\varphi\in\Omega_\mathcal{O}^p(M).
\end{equation}
It is easy to see that $\partial^{(k)}_f\circ\partial_f^{(k)}=0$, and we denote by $H^\bullet_{f,k}(M)$ the cohomology of the differential complex $(\Omega^\bullet_\mathcal{O}(M),\partial_f^{(k)})$,  which is called the \textit{holomorphic cohomology groups  of $M$ attached to the function $f$ and to the integer $k$}. This cohomology was considered for the first time in \cite{Mo1} in the context of Poisson geometry, and more generally, Nambu-Poisson geometry.

Using \eqref{II1} and \eqref{II2} we obtain
\begin{proposition}
 A holomorphic function $\alpha\in\mathcal{O}(M)$ is a holomorphic last multiplier for $Z\in\mathcal{X}_\mathcal{O}(M)$ if and only if $\theta$ is $\partial_\alpha^{(n)}$-closed.
\end{proposition}
Moreover, if we take $M=\mathbb{C}^n$, then in \cite{Mo1, Mo2} it is shown that $H^{n-1}_{f,n}(\mathbb{C}^n)=0$. Thus, if we consider $\omega=dz^1\wedge\ldots\wedge dz^n$ the standard volume form on $\mathbb{C}^n$ we have
\begin{proposition}
A holomorphic function $\alpha\in\mathcal{O}(\mathbb{C}^n)$ is a holomorphic last multiplier for $Z\in\mathcal{X}_\mathcal{O}(\mathbb{C}^n)$ if and only if there exists $\eta\in\Omega^{n-2}_\mathcal{O}(\mathbb{C}^n)$ such that $\theta=\partial_\alpha^{(n)}\eta$.
\end{proposition}

Now, for every $\alpha\in\mathcal{O}(M)$ we have $\partial\alpha\wedge\omega=0$. Then, for every $Z\in\mathcal{X}_\mathcal{O}(M)$  it results
\begin{displaymath}
0=\imath_Z(\partial\alpha\wedge\omega)=(\imath_Z\partial\alpha)\cdot\omega-\partial\alpha\wedge(\imath_Z\omega),
\end{displaymath}
or equivalently,
\begin{displaymath}
Z(\alpha)\cdot\omega=\partial\alpha\wedge(\imath_Z\omega)=\partial\alpha\wedge\theta.
\end{displaymath}
Thus, using \eqref{I3} and \eqref{II1} we obtain
\begin{proposition}
A holomorphic function $\alpha\in\mathcal{O}(M)$ is a holomorphic last multiplier for $Z\in\mathcal{X}_\mathcal{O}(M)$ if and only if
\begin{equation}
\label{II2}
Z(\alpha)+\alpha\cdot{\rm div}_\omega(Z)=0.
\end{equation}
\end{proposition}
\begin{example}
Let us consider an $n$-dimensional complex manifold $M$ endowed with a holomorphic volume form $\omega$ and let $Z\in\mathcal{X}_\mathcal{O}(M)$ be a holomorphic polynomial vector field on $M$. We recall that a holomorphic polynomial function $f\in\mathcal{O}(M)$ is called {\it a Darboux polynomial} for $Z$ if there is $g\in\mathcal{O}(M)$ such that $Z(f)=g\cdot f$, see \cite{I-Y}. The holomorphic function $g$ is said to be the cofactor corresponding to such holomorphic Darboux polynomial. Now, if $f_1,\ldots,f_p$ are holomorphic Darboux polynomials for $Z$ with corresponding holomorphic cofactors $g_k$, $k=1,\ldots,p$, then one can look for a holomorphic last multiplier for $Z$ of the form
\begin{displaymath}
\alpha=\prod_{k=1}^pf_k^{m_k}\,,\,m_k\in\mathbb{C}.
\end{displaymath}
Then, we have
\begin{displaymath}
\frac{Z(\alpha)}{\alpha}=\sum_{k=1}^pm_k\frac{Z(f_k)}{f_k}=\sum_{k=1}^pm_kg_k,
\end{displaymath}
and therefore, if the complex constants $m_k$ can be chosen such that $\sum\limits_{k=1}^pm_kg_k=-{\rm div}_\omega(Z)$, then, according to \eqref{II2}$, \alpha$ is a holomorphic last multiplier for $Z$.

\end{example}
Let us make some remarks concerning the importance of the relation \eqref{II2}.
\begin{enumerate}

\item[(i)] By \eqref{II2}, we see that a function $f\in\mathcal{O}(M)$ is last multiplier for the divergenceless holomorphic vector field $Z$ if and only if $\alpha\in I^1_\mathcal{O}(Z)$. The importance of this result is shown by the fact that three remarkable classes of divergence-free vector fields are provided by: Killing vector fields in Riemannian geometry, Hamiltonian vector fields in symplectic geometry and Reeb vector fields in contact geometry. Also, there are many equations of mathematical physics corresponding to the vector fields without divergence.
\item[(ii)] If $Z\in\mathcal{X}_\mathcal{O}(M)$ is not divergenceless, then we have the following relation between the holomorphic first integrals and the holomorphic last multipliers. Namely, from properties of Lie derivative, the ratio of two holomorphic last multipliers is a holomorphic first integral and conversely, the product between a holomorphic first integral and a holomorphic last multiplier is a holomorphic last multiplier. So, since $I^1_\mathcal{O}(Z)$ is a subalgebra in $\mathcal{O}(M)$ it results that the set of holomorphic last multipliers for $Z$ is a $I^1_\mathcal{O}(Z)$-module.
 \item[(iii)] The relations \eqref{I4} and \eqref{II2} say that $\alpha\in\mathcal{O}(M)$ is a holomorphic last multiplier for $Z\in\mathcal{X}_\mathcal{O}(M)$ if and only if ${\rm div}_\omega(\alpha Z)=0$. Thus, the set of holomorphic last multipliers is a "measure of how far away"  $Z$ is from being divergenceless.
\item[(iv)] To every holomorphic vector field $Z$ on $M$ we can associate an \textit{adjoint} $Z^*$, acting on $\mathcal{O}(M)$ by $Z^*(f)=-Z(f)-f{\rm div}_\omega(Z)$. Then, the set of holomorphic last multipliers of $Z$ coincides with $I^1_\mathcal{O}(Z^*)$.
\end{enumerate}

\begin{proposition}
Let $\alpha\in\mathcal{O}(M)$. The set of holomorphic vector fields for which $\alpha$ is a holomorphic last multiplier is a holomorphic Lie subalgebra in $\mathcal{X}_\mathcal{O}(M)$.
\end{proposition}
\begin{proof}
Let $Z,W\in\mathcal{X}_\mathcal{O}(M)$ such that $\alpha$ is a holomorphic last multiplier for both of them. Using \eqref{I2} and \eqref{II2} we have
\begin{eqnarray*}
[Z,W](\alpha)+\alpha {\rm div}_\omega([Z,W])&=&Z(W(\alpha))+\alpha Z({\rm div}_\omega(W))-W(Z(\alpha))-\alpha W({\rm div}_\omega(Z))\\
&=&Z(W(\alpha))-\alpha Z\left(W(\alpha)/\alpha\right)-W(Z(\alpha))+\alpha W\left(Z(\alpha)/\alpha\right)\\
&=&0.
\end{eqnarray*}
\end{proof}
Now, we search for a holomorphic last multiplier for $Z\in\mathcal{X}_\mathcal{O}(M)$ of divergence type, that is $\alpha={\rm div}_\omega(W)$ for some $W\in\mathcal{X}_\mathcal{O}(M)$. Using \eqref{II2} it results
\begin{equation}
\label{II4}
Z({\rm div}_\omega(W))+{\rm div}_\omega(W)\cdot {\rm div}_\omega(Z)=0.
\end{equation}
Multiplying \eqref{II4} by $\omega$ we have
\begin{displaymath}
\mathcal{L}_Z({\rm div}_\omega(W))\cdot\omega+{\rm div}_\omega(W)\cdot\mathcal{L}_Z\omega=0,
\end{displaymath}
or equivalently
\begin{displaymath}
\mathcal{L}_Z({\rm div}_\omega(W)\cdot\omega)=\mathcal{L}_Z\mathcal{L}_W\omega=0.
\end{displaymath}
Thus, we have
\begin{proposition}
If $W\in\mathcal{X}_\mathcal{O}(M)$ satisfies $\mathcal{L}_Z\mathcal{L}_W\omega=0$ then $\alpha={\rm div}_\omega(W)$ is a holomorphic last multiplier for $Z\in\mathcal{X}_\mathcal{O}(M)$.
\end{proposition}

Although the study of holomorphic last multipliers on complex manifolds seems to be identically with the study of real last multipliers on smooth manifolds, in the end of this subsection we present briefly our motivation for their study, and how a holomorphic last multiplier defines a real last multiplier for the associated real ODE system. 

Starting from the equivalence between a holomorphic ODE $dz/dt=F(z)$ and a real ODE system $dx/dt=U(x,y)$ , $dy/dt=V(x,y)$, where $z=x+iy$ and $F(z)=U(x,y)+iV(x,y)$ is a holomorphic function, the holomorphic ODE system \eqref{comsys} is equivalent with the real ODE system
\begin{equation}
\label{realsys}
\left\{
\begin{array}{ll}
\frac{dx^k}{dt}=X^k(x^{1}(t),\ldots,x^{n}(t),y^1(t),\ldots,y^n(t))&  \\
 &  \\
\frac{dy^{k}}{dt}=Y^{k}(x^1(t),\ldots,x^{n}(t),y^1(t),\ldots,y^n(t))&
\end{array}
\right., \,t\in\mathbb{R}
\end{equation}
where $z^k(t)=x^k(t)+iy^{k}(t)$, $k=1,\ldots,n$ and $Z^k(z)=X^k(x,y)+iY^k(x,y)$. The above real ODE system is canonically associated with the real vector field $Z_{\mathbb{R}}=2{\rm Re}\,Z=Z+\overline{Z}$, where overlines denotes the complex conjugation. Another canonically associated real vector field with $Z$ is $W_\mathbb{R}=2{\rm Im}\,Z=-i(Z-\overline{Z})$.
 
Now, if $\omega\in\Omega^n_\mathcal{O}(M)$ is a holomorphic volume form on the $n$-dimensional complex manifold $M$, then it is well know that $\omega_{\mathbb{R}}=\omega\wedge\overline{\omega}\in\Omega^{2n}(M)$ is a total real volume form on the underlying real manifold $M$, and we are interested if a holomorphic last multiplier for the holomorphic vector fied $Z$ defines a real last multiplier for the real vector field $Z_\mathbb{R}$ or $W_\mathbb{R}$ and conversely. In fact, we have
\begin{theorem}
\label{thlm}
A holomorphic function $\alpha\in\mathcal{O}(M)$ is a holomorphic last multiplier for the holomorphic vector field $Z\in\mathcal{X}_\mathcal{O}(M)$ if and only if $|\alpha|^2\in C^\infty(M;\mathbb{R})$ is a real last multiplier for both associated real vector fields $Z_\mathbb{R}$ and $W_\mathbb{R}$. 
\end{theorem}
\begin{proof}
Firstly, by direct computation we have
\begin{eqnarray*}
\theta_\mathbb{R}&=&\imath_{Z_\mathbb{R}}\omega_{\mathbb{R}}=\imath_Z(\omega\wedge\overline{\omega})+\imath_{\overline{Z}}(\omega\wedge\overline{\omega})=\theta\wedge\overline{\omega}+(-1)^n\omega\wedge\overline{\theta}.
\end{eqnarray*}
Now, according to the study of real last multipliers on smooth manifolds (see \cite{C2, C3}), $|\alpha|^2$ is a real last multiplier for the real vector field $Z_{\mathbb{R}}$ if $d(|\alpha|^2\theta_\mathbb{R})=0$. By direct computation we have
\begin{eqnarray*}
d(|\alpha|^2\theta_\mathbb{R})&=&(\partial+\overline{\partial})(\alpha\overline{\alpha})\wedge\theta_\mathbb{R}+\alpha\overline{\alpha}(\partial+\overline{\partial})\theta_\mathbb{R}\\
&=&(\overline{\alpha}\partial\alpha+\alpha\overline{\partial}\overline{\alpha})\wedge(\theta\wedge\overline{\omega}+(-1)^n\omega\wedge\overline{\theta})+\alpha\overline{\alpha}(\partial\theta\wedge\overline{\omega}+\omega\wedge\overline{\partial}\,\overline{\theta})\\
&=&\overline{\alpha}\partial\alpha\wedge\theta\wedge\overline{\omega}+\alpha\overline{\alpha}\partial\theta\wedge\overline{\omega}+\alpha\omega\wedge\overline{\partial}\overline{\alpha}\wedge\overline{\theta}+\alpha\overline{\alpha}\omega\wedge\overline{\partial}\,\overline{\theta}.
\end{eqnarray*}
But, $\partial\theta={\rm div}_\omega(Z)\omega$ and $\partial\alpha\wedge\theta=Z(\alpha)\omega$. Then, we have
\begin{eqnarray*}
d(|\alpha|^2\theta_\mathbb{R})&=&\left[\overline{\alpha}(Z(\alpha)+\alpha{\rm div}_\omega(Z))+\alpha(\overline{Z(\alpha)+\alpha{\rm div}_\omega(Z)})\right]\omega\wedge\overline{\omega}.
\end{eqnarray*}
Thus, according with \eqref{II2}, if $\alpha$ is a holomorphic last multiplier for the holomorphic vector field $Z$ then $d(|\alpha|^2\theta_\mathbb{R})=0$, that is $|\alpha|^2$ is a real last multiplier for $Z_\mathbb{R}$.

Now, if we put $\eta_\mathbb{R}=\imath_{W_\mathbb{R}}\omega_\mathbb{R}$, we obtain
\begin{displaymath}
\eta_\mathbb{R}=-i\left[\theta\wedge\overline{\omega}-(-1)^n\omega\wedge\overline{\theta}\right],
\end{displaymath}
and a similar computation as above yields
\begin{displaymath}
d(|\alpha|^2\eta_\mathbb{R})=-i\left[\overline{\alpha}(Z(\alpha)+\alpha{\rm div}_\omega(Z))-\alpha(\overline{Z(\alpha)+\alpha{\rm div}_\omega(Z)})\right]\omega\wedge\overline{\omega}.
\end{displaymath}
The above relation implies according with \eqref{II2} that if $\alpha$ is holomorphic last multiplier for $Z$ then $|\alpha|^2$ is a real last multiplier for $W_\mathbb{R}$.

Conversely, from the above computations, we have that if $|\alpha|^2$ is a real last multiplier for both real vector fields $Z_\mathbb{R}$ and $W_\mathbb{R}$ then 
\begin{equation}
\label{hlm1}
{\rm Re}\,\left[\overline{\alpha}(Z(\alpha)+\alpha{\rm div}_\omega(Z))\right]=0\,\,\,{\rm and}\,\,\,{\rm Im}\,\left[\overline{\alpha}(Z(\alpha)+\alpha{\rm div}_\omega(Z))\right]=0,
\end{equation}
that is $\overline{\alpha}(Z(\alpha)+\alpha{\rm div}_\omega(Z))=0$. Since $\alpha\neq0$, this is just \eqref{II2} and the proof is finished.
\end{proof}

\subsection{Holomorphic last multipliers for gradient vector fields on holomorphic Riemannian manifolds}
In this subsection we study holomorphic last multipliers for holomorphic gradient vector fields of a complex manifold $M$ endowed with a holomorphic Riemannian metric $g$.

If $f\in\mathcal{O}(M)$ then, as usual, we define the holomorphic gradient vector field of $f$ by
\begin{equation}
\label{II5}
g(W,{\rm grad}f)=W(f),\,\forall\,Z\in\mathcal{X}_\mathcal{O}(M),
\end{equation}
and the Laplace operator for  $f\in\mathcal{O}(M)$ by
\begin{equation}
\label{II6}
\Delta f=({\rm div}_{\omega_g}\circ {\rm grad})f.
\end{equation}
Now, if $\alpha\in\mathcal{O}(M)$ is a holomorphic last multiplier for $Z={\rm grad}f$, the relation \eqref{II2} becomes
\begin{equation}
\label{II7}
g({\rm grad}f,{\rm grad}\alpha)+\alpha\Delta f=0.
\end{equation}
Using a straightforward computation in local complex coordinates on $M$, the following relation (similar to the smooth case)  also holds for holomorphic Riemannian manifolds
\begin{equation}
\label{II8}
g({\rm grad}f,{\rm grad}\alpha)=\frac{1}{2}(\Delta(f\alpha)-f\Delta\alpha-\alpha\Delta f).
\end{equation}
Hence, we obtain
\begin{equation}
\label{II9}
\Delta(f\alpha)+\alpha\Delta f=f\Delta \alpha.
\end{equation}
The last equation leads to
\begin{proposition}
Let $(M,g)$ be a holomorphic Riemannian manifold and $f,\alpha\in\mathcal{O}(M)$ such that $f$ is a holomorphic last multiplier for ${\rm grad}\alpha$ and $\alpha$ is a holomorphic last multiplier for ${\rm grad}f$. Then $f\alpha$ is a holomorphic harmonic function on $(M,g)$.
\end{proposition}
\begin{corollary}
Let $(M,g)$ be a holomorphic Riemannian manifold and $\alpha\in\mathcal{O}(M)$. Then $\alpha$ is a holomorphic last multiplier for $Z={\rm grad}\alpha$ if and only if $\alpha^2$ is a holomorphic harmonic function on $(M,g)$.
\end{corollary}
\begin{corollary}
Let $(M,g)$ be a holomorphic Riemannian manifold and $\alpha\in\mathcal{O}(M)$. Then $\alpha^2$ is a holomorphic harmonic function on $(M,g)$ if and only if
\begin{displaymath}
\alpha\Delta\alpha+g({\rm grad}\alpha,{\rm grad}\alpha)=0.
\end{displaymath}
\end{corollary}
According to \cite{LeB}, if $g$ is replaced with $\widetilde{g}=f\cdot g$, where $f$ is a non-vanishing holomorphic function on $M$, then $\omega_{\widetilde{g}}=f^{\frac{n}{2}}\omega_g$ is also a holomorphic volume form on $(M,\widetilde{g})$. Then, using \eqref{I1}, by direct computation we get
\begin{equation}
\label{II10}
{\rm div}_{\omega_{\widetilde{g}}}(Z)={\rm div}_{\omega_g}(Z)+\frac{n}{2}Z(\log f).
\end{equation}
Thus, from \eqref{II2} and \eqref{II10} we obtain
\begin{proposition}
Let $(M,g)$ be a holomorphic Riemannian manifold. The holomorphic last multipliers for $Z\in\mathcal{X}_\mathcal{O}(M)$ with respect to $g$ coincide with the holomorphic last multipliers for $Z\in\mathcal{X}_\mathcal{O}(M)$ with respect to $\widetilde{g}=f\cdot g$ if and only if $\log f\in I^1_\mathcal{O}(Z)$.
\end{proposition}

Now, using Theorem \ref{thlm}, we can relate the holomorphic last multipliers of a holomorphic vector field of gradient type associated with a holomorphic Riemannian metric with real last multipliers of two vector fields of gradient type associated with anti-K\"{a}hlerian metrics defined by real and imaginary parts of the holomorphic Riemannian metric.

According to \cite{B-F-V, Sl}, there is an one-to-one correspondence between holomorphic Riemannian metrics on the complex manifold $M$ and anti-K\"{a}hlerian metrics on the underlying real manifold $(M,J)$. More exactly, if we consider the local complex coordinates $(z^1=x^1+ix^{n+1},\ldots,z^n=x^n+ix^{2n})$ in a local chart of $M$ and $g=(g_{ij})_{n\times n}$ is a holomorphic Riemannian metric on the complex manifold $M$,  then 
\begin{equation}
\label{r1}
g_{ij}=\frac{1}{2}(h_{ij}-ik_{ij})\,,\,g_{ij}=g\left(\frac{\partial}{\partial z^i},\frac{\partial}{\partial z^j}\right)\,,\,h_{ij}=h\left(\frac{\partial}{\partial x^i},\frac{\partial}{\partial x^j}\right)\,,\,k_{ij}=h_{in+j}=h\left(\frac{\partial}{\partial x^i},\frac{\partial}{\partial x^{n+j}}\right),
\end{equation}
where
\begin{displaymath}
\frac{\partial}{\partial z^j}=\frac{1}{2}\left(\frac{\partial}{\partial x^j}-i\frac{\partial}{\partial x^{n+j}}\right)\,,\,\frac{\partial}{\partial \overline{z}^j}=\frac{1}{2}\left(\frac{\partial}{\partial x^j}+i\frac{\partial}{\partial x^{n+j}}\right).
\end{displaymath}
Here the real part $h$ is an anti-K\"{a}hlerian metric on $(M,J)$ and the imaginary part $k$ is the associated anti-K\"{a}hlerian twin metric defined by $k(X,Y)=h(JX,Y)=h(X,JY)$. Also, the following relations hold:
\begin{equation}
\label{r2}
h_{ij}=-h_{n+in+j}=2{\rm Re}\,g_{ij}=g_{ij}+\overline{g_{ij}}\,,\,h_{n+ij}=h_{in+j}=-2{\rm Im}\,g_{ij}=i(g_{ij}-\overline{g_{ij}})\,,\,
\end{equation}
\begin{equation}
\label{r3}
k_{ij}=h_{in+j}=-2{\rm Im}\,g_{ij}=i(g_{ij}-\overline{g_{ij}})\,,\,k_{n+ij}=k_{in+j}=2{\rm Re}\,g_{ij}=g_{ij}+\overline{g_{ij}}.
\end{equation}
Moreover, if $(h^{\cdot\,\cdot})_{2n\times 2n}$ and $(k^{\cdot\,\cdot})_{2n\times 2n}$ denotes the inverse matrices of $(h_{\cdot\,\cdot})_{2n\times 2n}$ and $(k_{\cdot\,\cdot})_{2n\times 2n}$, respectively, then it is easy to see that
\begin{equation}
\label{r4}
h^{ij}=-h^{n+in+j}={\rm Re}\,g^{ij}\,,\,h^{n+ij}=h^{in+j}={\rm Im}\,g^{ij}
\end{equation}
and
\begin{equation}
\label{r5}
k^{ij}=-k^{n+in+j}={\rm Im}\,g^{ij}\,,\,k^{n+ij}=k^{in+j}=-{\rm Re}\,g^{ij}.
\end{equation} 
Let us denote by ${\rm grad}_hu$ and ${\rm grad}_ku$ the gradient vector fields of a smooth function $u$ with respect to real metrics $h$ and $k$, respectively. We have
\begin{theorem}
\label{tg}
Let $f\in\mathcal{O}(M)$. Then $\alpha$ is a holomorphic last multiplier for ${\rm grad}_g\,(\log f)$ if and only if $|\alpha|^2$ is a real last multiplier for both vector fields $(1/|f|^2){\rm grad}_h|f|^2$ and $(1/|f|^2){\rm grad}_k|f|^2$.
\end{theorem}
\begin{proof}
With respect to the complex coordinates $(z^1,\ldots,z^n)$ on $M$ the local form of the holomorphic gradient vector field ${\rm grad}_g(\log f)$ (corresponding to the holomorphic Riemannian metric $g$) is
\begin{equation}
\label{r6}
{\rm grad}_g(\log f)=\frac{1}{f}\sum_{i,j=1}^ng^{ij}\frac{\partial f}{\partial z^i}\frac{\partial}{\partial z^j}
\end{equation}
and with respect to real coordinates $(x^1,\ldots,x^{2n})$ the local form of the gradient vector field ${\rm grad}_h(|f|^2)$ (corresponding to the real metric $h$) is
\begin{equation}
\label{r7}
{\rm grad}_h|f|^2=\sum_{i,j=1}^nh^{ij}\frac{\partial|f|^2}{\partial x^i}\frac{\partial}{\partial x^j}+\sum_{i,j=1}^nh^{in+j}\frac{\partial|f|^2}{\partial x^i}\frac{\partial}{\partial x^{n+j}}
\end{equation}
\begin{displaymath}
\,\,\,\,\,\,\,\,\,\,\,\,\,\,\,\,\,\,\,\,\,\,\,\,\,\,\,\,\,\,\,\,\,\,\,\,\,\,\,\,\,\,\,\,\,\,\,\,+\sum_{i,j=1}^nh^{n+ij}\frac{\partial|f|^2}{\partial x^{n+i}}\frac{\partial}{\partial x^j}+\sum_{i,j=1}^nh^{n+in+j}\frac{\partial|f|^2}{\partial x^{n+i}}\frac{\partial}{\partial x^{n+j}},
\end{displaymath}
and the local form of the gradient vector field ${\rm grad}_k(|f|^2)$ (corresponding to the real metric $k$) is
\begin{equation}
\label{r8}
{\rm grad}_k|f|^2=\sum_{i,j=1}^nk^{ij}\frac{\partial|f|^2}{\partial x^i}\frac{\partial}{\partial x^j}+\sum_{i,j=1}^nk^{in+j}\frac{\partial|f|^2}{\partial x^i}\frac{\partial}{\partial x^{n+j}}
\end{equation}
\begin{displaymath}
\,\,\,\,\,\,\,\,\,\,\,\,\,\,\,\,\,\,\,\,\,\,\,\,\,\,\,\,\,\,\,\,\,\,\,\,\,\,\,\,\,\,\,\,\,\,\,\,+\sum_{i,j=1}^nk^{n+ij}\frac{\partial|f|^2}{\partial x^{n+i}}\frac{\partial}{\partial x^j}+\sum_{i,j=1}^nk^{n+in+j}\frac{\partial|f|^2}{\partial x^{n+i}}\frac{\partial}{\partial x^{n+j}}.
\end{displaymath}
Then, using \eqref{r4}, \eqref{r5} and 
\begin{displaymath}
\frac{\partial}{\partial x^k}=\frac{\partial}{\partial z^k}+\frac{\partial}{\partial \overline{z}^k}\,,\,\frac{\partial}{\partial x^{n+k}}=i\left(\frac{\partial}{\partial z^k}-\frac{\partial}{\partial \overline{z}^k}\right),
\end{displaymath}
a straightforward computation in the realations \eqref{r7} and \eqref{r8} yields
\begin{equation}
\label{r9}
\frac{1}{|f|^2}{\rm grad}_h|f|^2={\rm grad}_g(\log f)+\overline{{\rm grad}_g(\log f)}
\end{equation}
and 
\begin{equation}
\label{r10}
\frac{1}{|f|^2}{\rm grad}_k|f|^2=-i\left[{\rm grad}_g(\log f)-\overline{{\rm grad}_g(\log f)}\right].
\end{equation}
Then the proof follows by Theorem \ref{thlm}.
\end{proof}

\subsection{Inverse holomorphic multipliers}
 The relation \eqref{II2} says that if $0\neq\beta\in\mathcal{O}(M)$ satisfies the equation
\begin{equation}
\label{II3}
\mathcal{L}_Z\beta:=Z(\beta)=({\rm div}_\omega(Z))\cdot \beta,
\end{equation}
then $1/\beta$ is a holomorphic last multiplier for $Z$. Hence, $\beta\in\mathcal{O}(M)$ which satisfies \eqref{II3} will be called an \textit{inverse holomorphic multiplier} for $Z$.

Let us recall that a holomorphic vector field $S\in\mathcal{X}_\mathcal{O}(M)$ is said to be a \textit{symmetry} of $Z\in\mathcal{X}_\mathcal{O}(M)$ if there exists $\lambda\in\mathcal{O}(M)$ such that $\mathcal{L}_ZS:=[Z,S]=\lambda Z$. Consequently, if we consider $n-1$ symmetries $S_1,\ldots,S_{n-1}$ of $Z$, and we define
$\beta=\imath_{S_{n-1}}\ldots\imath_{S_1}\theta$, then $\beta$ is an inverse holomorphic multiplier for $Z$. This can be proved using the symmetry condition. Indeed
\begin{displaymath}
\mathcal{L}_Z\beta=\mathcal{L}_Z\imath_{S_{n-1}}\ldots\imath_{S_1}\theta=\left(\imath_{[Z,S_{n-1}]}+\imath_{S_{n-1}}\mathcal{L}_Z\right)\imath_{S_{n-2}}\ldots\imath_{S_1}\theta.
\end{displaymath}
The first term in the above expression vanishes, and recursively, it follows $\mathcal{L}_Z\beta=\beta\cdot {\rm div}_\omega(Z)$.

Another characterization of inverse holomorphic multiplier for $Z\in\mathcal{X}_\mathcal{O}(M)$ can be given in the following theorem which is a holomorphic version of Theorem 10 from \cite{B-G}.
\begin{theorem}
\label{ti}
Let $M$ be a $n$-dimensional complex manifold endowed with a holomorphic volume form $\omega$ and $Z\in\mathcal{X}_\mathcal{O}(M)$. If there exists a holomorphic frame field $\{Z_1,\ldots,Z_n\}$ of  $\mathcal{X}_\mathcal{O}(M)$ such that
\begin{equation}
\label{i1}
[Z,Z_i]=\sum_{k=1}^nf_i^kZ_k,
\end{equation}
 where $f_i^k\in\mathcal{O}(M)$, $i,k=1,\ldots,n$ satisfies ${\rm Tr}(f_i^k):=\sum\limits_{k=1}^nf_k^k=0$, then $\beta=\omega(Z_1,\ldots,Z_n)$ is an inverse holomorphic multiplier for $Z$.
\end{theorem}
\begin{proof}
It follows by direct computation involving the formula of Lie derivative and \eqref{I1}.
\end{proof}
Also, we have
\begin{corollary}
Let $M$ be an $n$-dimensional complex manifold endowed with a holomorphic volume form $\omega$ and $\{Z_1,\ldots,Z_n\}$ a holomorphic frame field of  $\mathcal{X}_\mathcal{O}(M)$ such that $[Z_i,Z_j]=\sum\limits_{k=1}^nf_{ij}^kZ_k$, where $f_{ij}^k\in\mathcal{O}(M)$. If there exist $g_k\in\mathcal{O}(M)$, $k=1,\ldots,n$ such that
\begin{equation}
\label{i2}
\sum_{k=1}^n\left(\sum_{i=1}^ng_if_{ik}^k-Z_k(g_k)\right)=0,
\end{equation}
then $\beta=\omega(Z_1,\ldots,Z_n)$ is an inverse holomorphic multiplier for $Z=\sum\limits_{k=1}^ng_kZ_k$.
\end{corollary}
\begin{proof}
By direct computation, we get
\begin{displaymath}
[Z,Z_j]=\sum_{k=1}^n\left(\sum_{i=1}^ng_if_{ij}^k-Z_j(g_k)\right)Z_k
\end{displaymath}
and then the result follows from Theorem \ref{ti}.
\end{proof}
\begin{example}
Let us consider the standard holomorphic volume form $\omega=dz^1\wedge dz^2\wedge dz^3$ on $\mathbb{C}^3$ and the holomorphic vector field $Z=\sum\limits_{i=1}^3Z^i(\partial/\partial z^i)$ on $\mathbb{C}^3$, where
\begin{equation}
\label{i3}
Z^i=\sum_{j=1}^3a_{ij}z^iz^j,\,a_{ij}\in\mathbb{C},\,i=1,2,3.
\end{equation}
Using  Theorem \ref{ti}, we describe a method to obtain an inverse holomorphic multiplier for the holomorphic vector field $Z$. We choose  three holomorphic vector fields on $\mathbb{C}^3$ given by $Z_i=(z^i)^{c_i}(\partial/\partial z^i)$, $i=1,2,3$, where $c_i\in\mathbb{C}$. By direct computation, we obtain $[Z,Z_i]=\sum\limits_{k=1}^3f_i^kZ_k$, where $f_i^k(z)=0$ for $i\neq k$, and
\begin{displaymath}
f_i^i(z)=(c_i-1)\sum_{j=1}^3a_{ij}z^j-a_{ii}z^i,\,i=1,2,3.
\end{displaymath}
Similar computations as in \cite{B-G} imply that $\sum\limits_{i=1}^3f_i^i(z)=0$ if
\begin{displaymath}
c_i=1+\frac{\Delta_i}{\Delta},\,i=1,2,3
\end{displaymath}
where $\Delta=\det(a_{ij})$ and
\begin{displaymath}
\Delta_1=\det\left(\begin{array}{cccccc}a_{11} & a_{21} & a_{31}  \\
a_{22} & a_{22} & a_{32} \\
a_{33} & a_{23} &  a_{33}
\end{array}\right)\,,\,\Delta_2=\det\left(\begin{array}{cccccc}a_{11} & a_{11} & a_{31}  \\
a_{12} & a_{22} & a_{32} \\
a_{13} & a_{33} &  a_{33}
\end{array}\right)\,,\,\Delta_3=\det\left(\begin{array}{cccccc}a_{11} & a_{21} & a_{11}  \\
a_{12} & a_{22} & a_{22} \\
a_{13} & a_{23} &  a_{33}
\end{array}\right).
\end{displaymath}
Then, the holomorphic function $\beta=\omega(Z_1,Z_2,Z_3)=(z^1)^{c_1}(z^2)^{c_2}(z^3)^{c_3}$ is an inverse holomorphic multiplier for $Z$.

As an application of Theorem \ref{thlm} we obtain that 
\begin{displaymath}
\frac{1}{|\beta|^2}=\frac{1}{|(z^1)^{c_1}(z^2)^{c_2}(z^3)^{c_3}|^2}
\end{displaymath}
is an Jacobi integrating factor for the following six dimensional real ODE system
\begin{equation}
\label{erealsys}
\left\{
\begin{array}{ll}
\frac{dx^i}{dt}=\sum\limits_{j=1}^3\left[p_{ij}(x^ix^j-y^iy^j)-q_{ij}(x^iy^j+y^ix^j)\right]&  \\
 &  \\
\frac{dy^{i}}{dt}=\sum\limits_{j=1}^3\left[p_{ij}(x^iy^j+y^ix^j)+q_{ij}(x^ix^j-y^iy^j)]\right]&
\end{array}
\right.,\,i=1,2,3,
\end{equation}
where $p_{ij}={\rm Re}\,a_{ij}$ and $q_{ij}={\rm Im}\,a_{ij}$.
\end{example}
We end our discussion concerning inverse holomorphic multipliers on complex manifolds with the following proposition.
\begin{proposition}
Let $(M_i,\omega_i)$, $i=1,2$ be two complex manifolds of complex dimensions $n_1$ and $n_2$, respectively, endowed with the holomorphic volume forms $\omega_1$ and $\omega_2$, respectively. Then $M:=M_1\times M_2$ is a $(n_1+n_2)$-dimensional complex manifold endowed with the holomorphic volume form $\omega:=\omega_1\wedge\omega_2$. If $\beta_1\in\mathcal{O}(M_1)$ is an inverse holomorphic multiplier for $Z_1\in\mathcal{X}_\mathcal{O}(M_1)$ and $\beta_2\in\mathcal{O}(M_2)$ is an inverse holomorphic multiplier for $Z_2\in\mathcal{X}_\mathcal{O}(M_2)$ then $\beta:=\beta_1\cdot\beta_2\in\mathcal{O}(M)$ is an inverse holomorphic multiplier for $Z=Z_1+Z_2\in\mathcal{X}_\mathcal{O}(M)$.
\end{proposition}
\begin{proof}
According to our hypothesis, and using \eqref{II3}, we have
\begin{eqnarray*}
Z(\beta)&=&\beta_2\cdot Z_1(\beta_1)+\beta_1\cdot Z_2(\beta_2)\\
&=&\beta_2\cdot\beta_1\cdot {\rm div}_{\omega_1}(Z_1)+\beta_1\cdot\beta_2\cdot {\rm div}_{\omega_2}(Z_2).
\end{eqnarray*}
On the other hand, we have
\begin{eqnarray*}
\beta\cdot{\rm div}_{\omega}(Z)\omega&=&\beta\left(\mathcal{L}_{Z_1}\omega_1\wedge\omega_2+\omega_1\wedge\mathcal{L}_{Z_2}\omega_2\right)\\
&=& \beta({\rm div}_{\omega_1}(Z_1)+{\rm div}_{\omega_2}(Z_2))\omega
\end{eqnarray*}
which end the proof.
\end{proof}

\section{Holomorphic last multipliers on holomorphic Poisson manifolds}
\setcounter{equation}{0}
In this section, we consider the study of holomorphic last multipliers in the framework of holomorphic Poisson manifolds. The particular class of holomorphic Hamiltonian vector fields associated with holomorphic Poisson structures with a special view towards the unimodular case and some examples are considered. Also, we relate the holomorphic last multipliers for holomorphic Hamiltonian vector fields with real last multipliers for Hamiltonian vector fields associated to natural real Poisson structures on the underlying real manifold. Next, we extend holomorphic last multipliers from holomorphic vector fields to holomorphic multivectors and we study carefully the Poisson bivector as a remarkable example. Also, some examples are considered to ilustrate our theory.

\subsection{Holomorphic Poisson structures}
Holomorphic Poisson structures appear naturally in many places, \cite{Ba, Ev-Lu}. For instance, any semi-simple complex Lie group admits a natural Poisson group structure, which is holomorphic. Its dual is also a holomorphic Poisson group. Indeed,
one of the simplest types of examples of holomorphic Poisson manifolds are the Lie-Poisson structures on the dual of complex Lie algebras. We notice that a study of holomorphic Poisson structures on complex manifolds was initiated in \cite{P}. Here, following \cite{L-P-V, L-S-X, S} we briefly recall some basic notions concerning these structures. A \textit{holomorphic Poisson manifold} is a complex manifold $M$ whose sheaf of holomorphic functions $\mathcal{O}(M)$ is a sheaf of Poisson algebras. By a sheaf of Poisson algebras over $M$ we mean that, for each open subset $U\subset M$, the ring $\mathcal{O}(U)$ is endowed with a Poisson bracket such that all restriction maps $\mathcal{O}(U)\rightarrow \mathcal{O}(V)$ (for arbitrary open subsets $V\subset U\subset M$) are morphisms of Poisson algebras. Moreover, given an open subset $U\subset M$, an open covering $\{U_i\},\,i\in I$ of $U$, and a pair of functions $f,g\in\mathcal{O}(U)$, the local data $\{f|_{U_i}, g|_{U_i}\},\,i\in I$ glue up and give $\{f|_U, g|_U\}$ if they coincide on the overlaps $U_i\cap U_j$. On a given complex manifold $M$, holomorphic Poisson structures are in one-to-one correspondence with sections $P\in\Gamma\left(\bigwedge^2T^{1,0}M\right)$ such that $\overline{\partial}P=0$ and $[P,P] = 0$. Here $[\cdot,\cdot]$ is the Schouten-Nijenhuis-Poisson bracket (for holomorphic tensor fields, see for instance \cite{Ib}). The Poisson bracket on functions and bivector field are related by the formula $P(\partial f,\partial g)=\{f,g\}$, where $f,g\in\mathcal{O}(M)$. Also, for $f\in\mathcal{O}(M)$ the operator   $Z_f:\mathcal{O}(M)\rightarrow \mathcal{O}(M)\,:\,g\mapsto Z_fg=\{f,g\}$ defines a derivation on $\mathcal{O}(M)$, i.e., it is a holomorphic vector field on $M$, called the \textit{holomorphic Hamiltonian vector field} of $f$, and satisfies $Z_f=\imath_{\partial f}P$, see for instance \cite{B-Z}.

Let $(M,P)$ be a holomorphic Poisson manifold and $Z_f$ the holomorphic Hamiltonian vector field for $f\in\mathcal{O}(M)$. Assume that  $M$ admits a holomorphic volume form $\omega$. Then the operator
\begin{displaymath}
Z_\omega:f\in\mathcal{O}(M)\mapsto {\rm div}_\omega(Z_f)\in\mathcal{O}(M)
\end{displaymath}
is a derivation on $\mathcal{O}(M)$, i.e., it is a holomorphic vector field on $M$, called the \textit{holomorphic modular vector field} of $(M,P,\omega)$.

Let us denote by $\mathcal{V}_\mathcal{O}^p(M)$ the space of holomorphic $p$-vector fields on $M$, i.e., skew symmetric contravariant holomorphic tensor fields of type $(p,0)$ on $M$. The Lichnerowicz-Poisson coboundary operator on a holomorphic Poisson manifold $(M,P)$ is defined by
\begin{displaymath}
\sigma:=[P,\cdot]:\mathcal{V}_\mathcal{O}^p(M)\rightarrow\mathcal{V}_\mathcal{O}^{p+1}(M)
\end{displaymath}
where $[\cdot,\cdot]$ is the Schouten-Nijenhuis bracket, and the holomorphic Lichnerowicz-Poisson cohomology (HLP) of $(M,P)$ is defined as the cohomology of the complex $(\mathcal{V}_\mathcal{O}^\bullet(M),\sigma)$, see for instance \cite{C-F-I-U}.
Then, for a holomorphic modular vector field one has $\sigma(Z_\omega)=0$, so it defines an $1$-dimensional HLP-cohomology class $[Z_\omega]\in H^1_{\mathcal{O},LP}(M,P)$. It is easy to see that this class does not depend on the holomorphic volume form $\omega$. It is called the \textit{holomorphic modular class} of the holomorphic Poisson manifold $(M,P)$.

The holomorphic Poisson manifold $(M,P,\omega)$ is called \textit{unimodular}, see \cite{S}, if $Z_\omega$ is the holomorphic Hamiltonian vector field $Z_h$ of a given $h\in\mathcal{O}(M)$.

\subsection{Last multipliers for holomorphic Hamiltonian vector fields}

In what follows, we consider  $(M,P,\omega)$ a holomorphic Poisson manifold endowed with a holomorphic volume form $\omega$. If $\alpha\in\mathcal{O}(M)$ is a holomorphic last multiplier for the holomorphic Hamiltonian vector field $Z_f$, then from \eqref{II2} it results
\begin{displaymath}
0=Z_f(\alpha)+\alpha Z_\omega(f)=-Z_\alpha(f)+\alpha Z_\omega(f)
\end{displaymath}
which leads to
\begin{proposition}
\label{pIII1}
Let $f\in\mathcal{O}(M)$. Then $\alpha\in\mathcal{O}(M)$ is a holomorphic last multiplier for the holomorphic Hamiltonian vector field $Z_f$ if and only if $f\in I^1_\mathcal{O}(\alpha Z_\omega-Z_\alpha)$, where $Z_\alpha$ is the holomorphic Hamiltonian vector field of $\alpha$. In the case when $(M,P,\omega)$ is unimodular, then $\alpha$ is a holomorphic last multiplier for $Z_f$ if and only if $\alpha\{h,f\}=\{\alpha,f\}$.
\end{proposition}
Taking into account that $f\in I^1_\mathcal{O}(Z_f)$ we get
\begin{corollary}
\label{cIII1}
Let $f\in\mathcal{O}(M)$. Then $f$ is a holomorphic last multiplier for $Z_f$ if and only if $f\in I^1_\mathcal{O}(Z_\omega)$. In the unimodular case, $f$ is a holomorphic last multiplier for $Z_f$ if and only if $\{h,f\}=0$.
\end{corollary}
The Jacobi and Leibnitz formulas of the holomorphic Poisson bracket $\{\cdot,\cdot\}$ imply
\begin{proposition}
Let $(M,P,\omega)$ be an unimodular holomorphic Poisson manifold. Then the set of all holomorphic last multipliers for the holomorphic Hamiltonian vector field $Z_f$ of $f\in\mathcal{O}(M)$ is a Poisson subalgebra of $(\mathcal{O}(M),\{\cdot,\cdot\})$.
\end{proposition}

Another consequence of the Proposition \ref{pIII1} is
\begin{corollary}
If $\alpha\in\mathcal{O}(M)$ is a holomorphic last multiplier for the holomorphic Hamiltonian vector fields $Z_f$ and $Z_g$ of $f,g\in\mathcal{O}(M)$ then $\alpha$ is a holomorphic last multiplier for $Z_{fg}$. Then, if $\alpha$ is a holomorphic last multiplier for $Z_f$ then $\alpha$ is a holomorphic last multiplier for $Z_{f^r}$, $r\in\mathbb{N}^*$.
\end{corollary}

Let $(z^1,\ldots,z^n)$ be a local coordinates system on $M$ such that $\omega=dz^1\wedge\ldots\wedge dz^n$ and the holomorphic Poisson bivector of $(M,\{\cdot,\cdot\})$ is
\begin{displaymath}
P=\sum\limits_{i<j}^nP^{ij}\frac{\partial}{\partial z^i}\wedge\frac{\partial}{\partial z^j}.
\end{displaymath}

If we denote $P^i=\sum\limits_{j=1}^n(\partial P^{ij}/\partial z^j)$, then a standard computation similar to the smooth case, see for instance Ch. 2.6 in \cite{D-Z}, yields
\begin{equation}
\label{III1}
Z_\omega=\sum_{i=1}^nP^i\frac{\partial}{\partial z^i}.
\end{equation}
Then Proposition \ref{pIII1} and Corollary \ref{cIII1} have the following local form.
\begin{proposition}
Let $f\in\mathcal{O}(M)$. Then $\alpha\in\mathcal{O}(M)$ is a holomorphic last multiplier for the holomorphic Hamiltonian vector field $Z_f$ of $f$ if and only if
\begin{equation}
\label{III2}
\alpha\sum_{i=1}^n P^i\frac{\partial f}{\partial z^i}=\{\alpha,f\}=\sum_{i<j}^n P^{ij}\frac{\partial \alpha}{\partial z^i}\frac{\partial f}{\partial z^j},
\end{equation}
and $f$ is a holomorphic last multiplier for $Z_f$ if and only if
\begin{equation}
\label{III3}
\sum_{i=1}^n P^i\frac{\partial f}{\partial z^i}=0.
\end{equation}
\end{proposition}
\begin{example}
\label{e3.0}
Let $M$ be a $2$-dimensional complex manifold with local complex coordinates $(z^1,z^2)$ in a local complex chart of $M$ and with a holomorphic volume form $\omega=dz^1\wedge dz^2$. Then, according to \cite{S}, all holomorphic bivector fields $P=f(z^1,z^2)(\partial/\partial z^1)\wedge(\partial/\partial z^2)\in\mathcal{V}^2_\mathcal{O}(M)$ are automatically Poisson tensors. Now, if $\alpha\in\mathcal{O}(M)$ is a holomorphic last multiplier for the holomorphic Hamiltonian vector field $Z_\alpha$ of $\alpha$, then the equation \eqref{III3} reads as
\begin{displaymath}
\frac{\partial f}{\partial z^1}\frac{\partial \alpha}{\partial z^2}-\frac{\partial f}{\partial z^2}\frac{\partial \alpha}{\partial z^1}=0
\end{displaymath}
with the obvious solution $\alpha=\phi(f)$, with $\phi\in C^1(\mathbb{C})$. Thus, $f$ is a holomorphic last multiplier for its holomorphic Hamltonian vector field $Z_f$.

\end{example}
\begin{example}
\label{e3.1}
\textit{Holomorphic linear Poisson structure on $\mathbb{C}^n$}.  Following Example 1.3 in \cite{P} (see also Example 3.6 in \cite{C-F-I-U}) we consider the complex manifold $\mathbb{C}^n$, global complex coordinates $(z^1,\ldots,z^n)$, $\omega=dz^1\wedge\ldots\wedge dz^n$ and the holomorphic bivector field $P$ on $\mathbb{C}^n$ with the local components $P^{jk}$ holomorphic functions defined by $P^{jk}(z)=\sum\limits_{l=1}^nc^{jk}_lz^l$, where $c^{jk}_l$ are complex constants satisfying $\sum\limits_{l=1}^n(c^{ij}_lc^{lk}_h+c^{jk}_lc^{li}_h+c^{ki}_lc^{lj}_h)=0$ and $c^{jk}_l+c^{kj}_l=0$. Then, $P$ defines a holomorphic Poisson structure on $\mathbb{C}^n$ called a linear structure (or holomorphic Lie-Poisson structure). Hence, in this case $P^i=\sum\limits_{j=1}^nc^{ij}_j$, and \eqref{III3} says that $f\in\mathcal{O}(\mathbb{C}^n)$ is a holomorphic last multiplier for the holomorphic Hamiltonian vector field $Z_f$ if and only if
\begin{equation}
\label{III4}
\sum_{i,j=1}^nc^{ij}_j\frac{\partial f}{\partial z^i}=0,
\end{equation}
with the general solution
\begin{displaymath}
f=\phi\left(\sum_{j=1}^n(c^{2j}_jz^1-c^{1j}_jz^2),\sum_{j=1}^n(c^{3j}_jz^1-c^{1j}_jz^3),\ldots,\sum_{j=1}^n(c^{nj}_jz^1-c^{1j}_jz^n)\right),
\end{displaymath}
where $\phi\in C^1(\mathbb{C}^{n-1})$.
\end{example}
\begin{example}
\label{e3.2}
Take $\mathbb{C}^3$, with global complex coordinates $(z^1,z^2,z^3)$ and $\omega=dz^1\wedge dz^2\wedge dz^3$. Then it is easy to see that
\begin{equation}
\label{III6}
\{z^1,z^3\}=z^1z^3-2z^2\,,\,\{z^3,z^2\}=z^3z^2-2z^1\,,\,\{z^2,z^1\}=z^2z^1-2z^3
\end{equation}
defines a Poisson bracket on $\mathbb{C}^3$, and then $P^{ij}=\{z^i,z^j\}$ are the local components of a holomorphic Poisson bivector
$P$ on $\mathbb{C}^3$. By direct computation we obtain $P^i=0$ for every $i=1,2,3$, so that the equation \eqref{III3} is satisfied for every holomorphic function $f$ on $\mathbb{C}^3$. Thus, on the holomorphic Poisson manifold $(\mathbb{C}^3,P)$,  every $f\in\mathcal{O}(\mathbb{C}^3)$ is a holomorphic last multiplier for its holomorphic Hamiltonian vector field $Z_f$.
\end{example}

In what follows in this subsection we are interested how we can relate the holomorphic last multipliers for holomorphic Hamiltonian vector fields on holomorphic Poisson manifolds with some real last multipliers for real Hamiltonian vector fields corresponding to natural real Poisson structures on the underlying real manifold. 

According to \cite{L-S-X} if $(z^1=x^1+ix^{n+1},\ldots,z^{n}=x^n+ix^{2n})$ are local complex coordinates in a local chart on the $n$-dimensional complex manifold $M$ endowed with the holomorphic Poisson structure $P$ which is defined locally by the bracket $\{z^i,z^j\}=P^{ij}$, then the real and imaginary parts of $P=P_\mathbb{R}+iP_\mathbb{I}$ define on the underlying real manifold two real Poisson structures by
\begin{equation}
\label{rp}
\{x^i,x^j\}_\mathbb{R}=\frac{1}{4}{\rm Re}\,\{z^i,z^j\}\,,\,\{x^i,x^{n+j}\}_\mathbb{R}=\frac{1}{4}{\rm Im}\,\{z^i,z^j\}\,,\,\{x^{n+i},x^{n+j}\}_\mathbb{R}=-\frac{1}{4}{\rm Re}\,\{z^i,z^j\}
\end{equation}
and
\begin{equation}
\label{ip}
\{x^i,x^j\}_\mathbb{I}=\frac{1}{4}{\rm Im}\,\{z^i,z^j\}\,,\,\{x^i,x^{n+j}\}=-\frac{1}{4}{\rm Re}\,\{z^i,z^j\}\,,\,\{x^{n+i},x^{n+j}\}=-\frac{1}{4}{\rm Im}\,\{z^i,z^j\}.
\end{equation}
Let us denote by $Z^\mathbb{R}_{u}$ and $Z^\mathbb{I}_{u}$, respectively the real Hamiltonian vector fields of a smooth function $u$ on $M$ with respect to the real Poisson structures $P_\mathbb{R}$ and $P_\mathbb{I}$, respectively. We have
\begin{theorem}
\label{th1}
Let $f\in\mathcal{O}(M)$. Then $\alpha\in\mathcal{O}(M)$ is a holomorphic last multiplier for the Hamiltonian vector field $Z_{\log f}$ if and only if $|\alpha|^2$ is a real last multiplier for both real vector felds $(1/|f|^2)Z^\mathbb{R}_{|f|^2}$ and $(1/|f|^2)Z^\mathbb{I}_{|f|^2}$.
\end{theorem}
\begin{proof}
With respect to the complex coordinates $(z^1,\ldots,z^n)$ on $M$ the local form of the holomorphic Hamiltonian vector field $Z_{\log f}$ (corresponding to holomorphic Poisson structure $P$) is
\begin{equation}
\label{2016-1}
Z_{\log f}=\frac{1}{f}\sum_{i,j=1}^nP^{ij}\frac{\partial f}{\partial z^i}\frac{\partial}{\partial z^j}
\end{equation}
and with respect to real coordinates $(x^1,\ldots,x^{2n})$ the local form of the Hamiltonian vector field $Z^\mathbb{R}_{|f|^2}$ (corresponding to real Poisson structure $P_\mathbb{R}$) is
\begin{equation}
\label{2016-2}
Z^\mathbb{R}_{|f|^2}=\sum_{i,j=1}^nP_\mathbb{R}^{ij}\frac{\partial|f|^2}{\partial x^i}\frac{\partial}{\partial x^j}+\sum_{i,j=1}^nP_\mathbb{R}^{in+j}\frac{\partial|f|^2}{\partial x^i}\frac{\partial}{\partial x^{n+j}}
\end{equation}
\begin{displaymath}
\,\,\,\,\,\,\,\,\,\,\,\,\,\,\,\,\,\,\,\,\,\,\,\,\,\,\,\,\,\,\,\,\,\,\,\,\,\,\,\,\,\,\,\,\,\,\,\,+\sum_{i,j=1}^nP_\mathbb{R}^{n+ij}\frac{\partial|f|^2}{\partial x^{n+i}}\frac{\partial}{\partial x^j}+\sum_{i,j=1}^nP_\mathbb{R}^{n+in+j}\frac{\partial|f|^2}{\partial x^{n+i}}\frac{\partial}{\partial x^{n+j}},
\end{displaymath}
where $P_\mathbb{R}^{ij}=\{x^i,x^j\}_\mathbb{R}$, $P_\mathbb{R}^{in+j}=\{x^i,x^{n+j}\}_\mathbb{R}$, $P_\mathbb{R}^{n+ij}=-P_\mathbb{R}^{jn+i}$ and $P_\mathbb{R}^{n+in+j}=-P_\mathbb{R}^{ij}$, respectively.

Taking into account that
\begin{displaymath}
P_\mathbb{R}^{ij}=\frac{1}{4}{\rm Re}\,P^{ij}=\frac{P^{ij}+\overline{P^{ij}}}{8}\,,\,P_\mathbb{R}^{in+j}=\frac{1}{4}{\rm Im}\,P^{ij}=\frac{P^{ij}-\overline{P^{ij}}}{8i},
\end{displaymath}
\begin{displaymath}P_\mathbb{R}^{n+ij}=-P_\mathbb{R}^{jn+i}=-\frac{P^{ji}-\overline{P^{ji}}}{8i}\,,\,P_\mathbb{R}^{n+in+j}=-P_\mathbb{R}^{ij}=-\frac{P^{ij}+\overline{P^{ij}}}{8}
\end{displaymath}
and
\begin{displaymath}
\frac{\partial}{\partial x^k}=\frac{\partial}{\partial z^k}+\frac{\partial}{\partial \overline{z}^k}\,,\,\frac{\partial}{\partial x^{n+k}}=i\left(\frac{\partial}{\partial z^k}-\frac{\partial}{\partial \overline{z}^k}\right)
\end{displaymath}
a straightforward computation in \eqref{2016-2} yields 
\begin{displaymath}
Z^\mathbb{R}_{|f|^2}=\frac{1}{2}\left(\overline{f}\sum_{i,j=1}^nP^{ij}\frac{\partial f}{\partial z^i}\frac{\partial}{\partial z^j}+f\sum_{i,j=1}^n\overline{P^{ij}}\frac{\partial\overline{f}}{\partial\overline{z}^i}\frac{\partial}{\partial\overline{z}^j}\right),
\end{displaymath}
or equivalently $(2/|f|^2)Z^\mathbb{R}_{|f|^2}=Z_{\log f}+\overline{Z_{\log f}}=2{\rm Re}\,Z_{\log f}$.

Similarly, the local form of the Hamiltonian vector field $Z^\mathbb{I}_{|f|^2}$ (corresponding to real Poisson structure $P_\mathbb{I}$) is
\begin{equation}
\label{2016-3}
Z^\mathbb{I}_{|f|^2}=\sum_{i,j=1}^nP_\mathbb{I}^{ij}\frac{\partial|f|^2}{\partial x^i}\frac{\partial}{\partial x^j}+\sum_{i,j=1}^nP_\mathbb{I}^{in+j}\frac{\partial|f|^2}{\partial x^i}\frac{\partial}{\partial x^{n+j}}
\end{equation}
\begin{displaymath}
\,\,\,\,\,\,\,\,\,\,\,\,\,\,\,\,\,\,\,\,\,\,\,\,\,\,\,\,\,\,\,\,\,\,\,\,\,\,\,\,\,\,\,\,\,\,\,\,+\sum_{i,j=1}^nP_\mathbb{I}^{n+ij}\frac{\partial|f|^2}{\partial x^{n+i}}\frac{\partial}{\partial x^j}+\sum_{i,j=1}^nP_\mathbb{I}^{n+in+j}\frac{\partial|f|^2}{\partial x^{n+i}}\frac{\partial}{\partial x^{n+j}},
\end{displaymath}
where $P_\mathbb{I}^{ij}=\{x^i,x^j\}_\mathbb{I}$, $P_\mathbb{I}^{in+j}=\{x^i,x^{n+j}\}_\mathbb{I}$, $P_\mathbb{I}^{n+ij}=-P_\mathbb{I}^{jn+i}$ and $P_\mathbb{I}^{n+in+j}=-P_\mathbb{I}^{ij}$, respectively. Then a similar computation as above in \eqref{2016-3} yields
\begin{displaymath}
Z^\mathbb{I}_{|f|^2}=-\frac{i}{2}\left(\overline{f}\sum_{i,j=1}^nP^{ij}\frac{\partial f}{\partial z^i}\frac{\partial}{\partial z^j}-f\sum_{i,j=1}^n\overline{P^{ij}}\frac{\partial\overline{f}}{\partial\overline{z}^i}\frac{\partial}{\partial\overline{z}^j}\right),
\end{displaymath}
or equivalently $(2/|f|^2)Z^\mathbb{I}_{|f|^2}=-i\left(Z_{\log f}-\overline{Z_{\log f}}\right)=2{\rm Im}\,Z_{\log f}$.

Finally, the proof follows by Theorem \ref{thlm}.
\end{proof}

As well as we seen the relation \eqref{III3} says when a holomorphic function is a last multiplier for its holomorphic Hamiltonian vector field. This condition can be related in terms o real last multipliers for real Hamiltonian vector fields as follows.
\begin{theorem}
\label{th2}
A holomorphic function $f\in\mathcal{O}(M)$ is a last multiplier for its holomorphic Hamiltonian vector field $Z_f$ if and only if $|f|^2$ is a real last multiplier for both Hamiltonian vector fields $Z^\mathbb{R}_{|f|^2}$ and $Z^\mathbb{I}_{|f|^2}$.
\end{theorem}
\begin{proof}
If we take the real Poisson structure $P_\mathbb{R}$ we have
\begin{equation}
\label{2016-4}
P^i_{\mathbb{R}}=\sum_{k=1}^n\left(\frac{\partial P_\mathbb{R}^{ik}}{\partial x^k}+\frac{\partial P_\mathbb{R}^{in+k}}{\partial x^{n+k}}\right)=\frac{1}{4}\sum_{k=1}^n\left(\frac{\partial P^{ik}}{\partial z^k}+\frac{\partial \overline{P^{ik}}}{\partial \overline{z}^k}\right)
\end{equation}
and
\begin{equation}
\label{2016-5}
P^{n+i}_{\mathbb{R}}=\sum_{k=1}^n\left(\frac{\partial P_\mathbb{R}^{n+ik}}{\partial x^k}+\frac{\partial P_\mathbb{R}^{n+in+k}}{\partial x^{n+k}}\right)=\frac{i}{4}\sum_{k=1}^n\left(\frac{\partial P^{ki}}{\partial z^k}-\frac{\partial \overline{P^{ki}}}{\partial \overline{z}^k}\right).
\end{equation}
Then the relation \eqref{III3} (for the real case \cite{C2}) says that $|f|^2$ is real last multiplier for the Hamiltonian vector field $Z^\mathbb{R}_{|f|^2}$ if and only if
\begin{displaymath}
\sum_{i=1}^n\left(P_\mathbb{R}^i\frac{\partial|f|^2}{\partial x^i}+P_\mathbb{R}^{n+i}\frac{\partial|f|^2}{\partial x^{n+i}}\right)=0
\end{displaymath}
and, using \eqref{2016-4} and \eqref{2016-5}, this condition is equivalent with
\begin{equation}
\label{2016-6}
{\rm Re}\,\left[\overline{f}\sum_{i=1}^n\left(\sum_{k=1}^n\frac{\partial P^{ik}}{\partial z^k}\right)\frac{\partial f}{\partial z^i}\right]=0.
\end{equation}
Similarly, if we take the real Poisson structure $P_\mathbb{I}$ we have
\begin{equation}
\label{2016-7}
P^i_{\mathbb{I}}=\sum_{k=1}^n\left(\frac{\partial P_\mathbb{I}^{ik}}{\partial x^k}+\frac{\partial P_\mathbb{I}^{in+k}}{\partial x^{n+k}}\right)=-\frac{i}{4}\sum_{k=1}^n\left(\frac{\partial P^{ik}}{\partial z^k}-\frac{\partial \overline{P^{ik}}}{\partial \overline{z}^k}\right)
\end{equation}
and
\begin{equation}
\label{2016-8}
P^{n+i}_{\mathbb{I}}=\sum_{k=1}^n\left(\frac{\partial P_\mathbb{I}^{n+ik}}{\partial x^k}+\frac{\partial P_\mathbb{I}^{n+in+k}}{\partial x^{n+k}}\right)=-\frac{1}{4}\sum_{k=1}^n\left(\frac{\partial P^{ik}}{\partial z^k}+\frac{\partial \overline{P^{ik}}}{\partial \overline{z}^k}\right).
\end{equation}
Then $|f|^2$ is real last multiplier for the Hamiltonian vector field $Z^\mathbb{I}_{|f|^2}$ if and only if
\begin{displaymath}
\sum_{i=1}^n\left(P_\mathbb{I}^i\frac{\partial|f|^2}{\partial x^i}+P_\mathbb{I}^{n+i}\frac{\partial|f|^2}{\partial x^{n+i}}\right)=0
\end{displaymath}
and, using \eqref{2016-7} and \eqref{2016-8}, this condition is equivalent with
\begin{equation}
\label{2016-9}
{\rm Im}\,\left[\overline{f}\sum_{i=1}^n\left(\sum_{k=1}^n\frac{\partial P^{ik}}{\partial z^k}\right)\frac{\partial f}{\partial z^i}\right]=0.
\end{equation}
Now, the proof follows using \eqref{III3}, \eqref{2016-6} and \eqref{2016-9}.
\end{proof}

\subsection{Holomorphic last multipliers for holomorphic Poisson bivectors}

Let us consider as in the previous subsection $\mathcal{V}^p_\mathcal{O}(M)$ the $\mathcal{O}(M)$-module of holomorphic $p$-vector fields on $M$, $1\leq p\leq n$. A holomorphic $p$-vector field $A$ defines the map $\imath_A:\Omega^k_\mathcal{O}(M)\rightarrow\Omega^{k-p}_\mathcal{O}(M)$ given by
\begin{displaymath}
\langle\imath_A\varphi,B\rangle=\langle\varphi,A\wedge B\rangle
\end{displaymath}
for every $\varphi\in\Omega^p_\mathcal{O}(M)$ and $B\in\mathcal{V}^{k-p}_\mathcal{O}(M)$, $k\geq p$ where $\langle,\rangle$ is the natural duality between holomorphic forms and holomorphic multivectors. We note that $\imath_A\varphi=0$ for $k<p$.

Consider that $M$ is endowed with a holomorphic volume form $\omega\in\Omega^n_\mathcal{O}(M)$. Then $\omega$ defines the map
\begin{equation}
\label{IV1}
\omega^\flat:\mathcal{V}^p_\mathcal{O}(M)\rightarrow\Omega^{n-p}_\mathcal{O}(M)\,,\,\omega^\flat(A)=\imath_A\omega,
\end{equation}
which is an $\mathcal{O}(M)$-isomorphism between $\mathcal{V}^p_\mathcal{O}(M)$ and $\Omega^{n-p}_\mathcal{O}(M)$, for every $0\leq p\leq n$. For instance, if $(z^1,\ldots,z^n)$ are local complex coordinates on $M$, $A=A^{i_1\ldots i_p}(z)(\partial/\partial z^{i_1})\wedge\ldots\wedge(\partial/\partial z^{i_p})\in\mathcal{V}^p_\mathcal{O}(M)$, and $\omega=dz^1\wedge\ldots\wedge dz^n$, then
\begin{displaymath}
\omega^\flat(A)=\imath_A\omega=(-1)^{i_1-1}\ldots(-1)^{i_p-1}A^{i_1\ldots i_p}(z)dz^1\wedge\widehat{dz^{i_1}}\wedge\ldots\wedge\widehat{dz^{i_p}}\wedge\ldots\wedge dz^n.
\end{displaymath}
Also, we denote by $\omega^\sharp:\Omega^{n-p}_\mathcal{O}(M)\rightarrow\mathcal{V}^p_\mathcal{O}(M)$ the inverse map of $\omega^\flat$.
\begin{definition}
The map $D:\mathcal{V}^p_\mathcal{O}(M)\rightarrow\mathcal{V}^{p-1}_\mathcal{O}(M)$ given by
\begin{equation}
\label{IV2}
D_\omega=\omega^\sharp\circ \partial\circ\omega^\flat,
\end{equation}
is called the \textit{holomorphic curl operator} with respect to the holomorphic volume form $\omega$. If $A\in\mathcal{V}^p_\mathcal{O}(M)$ then $D_\omega A$ is called the \textit{holomorphic curl} of $A$.
\end{definition}
We notice that the above definition is considered in \cite{L-P} in the case $M=\mathbb{C}^n$.
\begin{definition}
A holomorphic $p$-multivector $A$ on $M$ is called \textit{exact} if $D_\omega A=0$. Specifically, a holomorphic Poisson bivector $P\in\mathcal{V}^2_\mathcal{O}(M)$ satisfying $D_\omega P=0$ is called an \textit{exact holomorphic Poisson structure}.
\end{definition}
Using a similar computation as in the smooth case (see Theorem 2.1 and Proposition 2.3 from \cite{Y}), we obtain the following characterization of the exactness of holomorphic Hamiltonian vector fields and of holomorphic Poisson bivector fields on $\mathbb{C}^n$.
\begin{proposition}
Let $P$ be a holomorphic Poisson structure on the holomorphic Riemannian manifold $(\mathbb{C}^n,g)$, where $g=\sum\limits_{j=1}^n(dz^j)^2$ is the holomorphic euclidian metric. Then a holomorphic Hamiltonian vector field $Z_h$ of $h\in\mathcal{O}(\mathbb{C}^n)$ is exact if and only if $g(D_\omega P,{\rm grad}h)=0$, or equivalently $h\in I^1_\mathcal{O}(D_\omega P)$.
\end{proposition}
\begin{proposition}
If $P\in\mathcal{V}^2_\mathcal{O}(\mathbb{C}^n)$ is  skew symmetric with the structure matrix $(P^{ij}(z))_{n\times n}$, then it is an exact holomorphic Poisson structure on $\mathbb{C}^n$ if and only if
\begin{equation}
\label{x1}
\sum_{j=1}^n\frac{\partial P^{ij}}{\partial z^j}=0,\,i=1,\ldots,n,
\end{equation}
and
\begin{equation}
\label{x2}
\sum_{l=1,s\neq i,j,k}^n\frac{\partial \left(A_l^{ijk}\cdot B^{ijk}\right)}{\partial z^l}=0,\,1\leq i<j<k\leq n,
\end{equation}
where $A_l^{ijk}=(P^{li},P^{lj},P^{lk})$ and $B^{ijk}=(P^{jk},P^{ki},P^{ij})$.
\end{proposition}
Another characterization of exactness of holomorphic bivector fields on $4$-dimensional complex manifolds can be given as in the smooth case, see \cite{C-M2}.
\begin{theorem}
Let $M$ be a $4$-dimensional complex manifold, $P\in\mathcal{V}^2_\mathcal{O}(M)$ which is skew-symmetric satisfying $P(z_0)=0$, $z_0\in M$, and $\omega$ a holomorphic volume form on $M$. Then the following assertions are equivalent:
\begin{enumerate}
\item[(i)] $P$ is an exact holomorphic Poisson bivector.
\item[(ii)] $\imath_P\omega\wedge\imath_P\omega=0$ and $\partial(\imath_P\omega)=0$.
\end{enumerate}
\end{theorem}
\begin{proposition}
Let $P$ be an exact holomorphic Poisson structure on a $4$-dimensional complex manifold $M$ satisfying $P(z_0)=0$, $z_0\in M$. Then we have
\begin{enumerate}
\item[(i)] The rank of $P$ at any point is at most two.
\item[(ii)] For every $\alpha\in\mathcal{O}(M)$, $\alpha P$ is a holomorphic Poisson structure on $M$.
\end{enumerate}
\end{proposition}
\begin{remark}
Let $P$ be a holomorphic Poisson structure on a $4$-dimensional complex manifold which is exact with respect to a holomorphic volume form $\alpha\omega$, where $\alpha$ is a non-vanishing holomorphic function on $M$. If $P(z_0)=0$, $z_0\in M$, then $\alpha P$ is also a holomorphic Poisson structure on $M$, and moreover, $\alpha P$ is exact with respect to $\omega$, since $\imath_{\alpha P}\omega=\imath_P(\alpha\omega)$.
\end{remark}

It is easy to see that $D_\omega\circ D_\omega=0$, so it is a homological operator, and if $\mathcal{V}_\mathcal{O}^\bullet(M)=\bigoplus\limits_{p=0}^n\mathcal{V}^p_\mathcal{O}(M)$ is the algebra given by the direct sum of the space of the holomorphic $p$-multivectors on $M$, then the homology of the differential complex $(\mathcal{V}_\mathcal{O}^\bullet(M),D_\omega)$ is given by
\begin{displaymath}
H_p(M)=\frac{{\rm Ker}\{D_\omega:\mathcal{V}^p_\mathcal{O}(M)\rightarrow\mathcal{V}^{p-1}_\mathcal{O}(M)\}}{{\rm Im}\{D_\omega:\mathcal{V}^{p+1}_\mathcal{O}(M)\rightarrow\mathcal{V}^{p}_\mathcal{O}(M)\}}.
\end{displaymath}
For instance, if $M=\mathbb{C}^n$, using the Poincar\'{e} Lemma for $d=\partial$, we get $H_p(\mathbb{C}^n)=0$, for every $0\leq p<n$ and $H_n(\mathbb{C}^n)=\mathbb{C}$. Thus, the differential complex $(\mathcal{V}_\mathcal{O}^\bullet(\mathbb{C}^n),D_\omega)$ is exact.

For $p=1$ we have $D_\omega={\rm div}_\omega$. Indeed, if $A\in\mathcal{X}_\mathcal{O}(M)$ then
\begin{equation}
\label{x3}
(D_\omega A)\omega=\omega^\flat\circ D_\omega(A)=\partial\circ\omega^\flat(A)=\partial\circ\imath_A(\omega)=\mathcal{L}_A\omega=({\rm div}_\omega(A))\omega.
\end{equation}

Taking into account\eqref{x3} and \eqref{I4}, we consider the following definition.
\begin{definition}
The function $\alpha\in\mathcal{O}(M)$ is called a \textit{holomorphic last multiplier} of $A\in\mathcal{V}^p_\mathcal{O}(M)$ if
\begin{equation}
\label{IV3}
D_\omega(\alpha A)=0,
\end{equation}
or equivalently, $\alpha A$ is an exact holomorphic $p$-vector on $M$.
\end{definition}
It  follows that the set of holomorphic last multipliers of $(M,P,\omega)$ is a "measure of how far away"  $(M,P,\omega)$ is from being exact.

Since $\omega^\sharp$ is an $\mathcal{O}(M)$-isomorphism between $\Omega^{n-p}_\mathcal{O}(M)$ and $\mathcal{V}^p_\mathcal{O}(M)$, it follows that \eqref{IV3} is equivalent with $\partial(\omega^\flat(\alpha A))=0$, that is
\begin{equation}
\label{IV4}
\partial(\alpha\omega^\flat(A))=0
\end{equation}
which is a natural extension of the condition \eqref{II1}.

From the $\mathcal{O}(M)$-linearity of $\omega^\flat$ we have $\omega^\flat(\alpha A)=\alpha\omega^\flat(A)=(\alpha\omega)^\flat(A)$ which implies $(\alpha\omega)^\flat=(1/\alpha)\omega^\flat$ (it is assumed that $\alpha\neq 0$ everywhere). It follows that
\begin{equation}
\label{IV5}
\alpha D_{\alpha\omega}(A)=\omega^\sharp\circ \partial\circ \omega^\flat(\alpha A)=D_\omega(\alpha A)
\end{equation}
which yields the following.
\begin{proposition}
The function $\alpha\in\mathcal{O}(M)$ is a holomorphic last multiplier for $A\in\mathcal{V}^p_\mathcal{O}(M)$ if and only if
\begin{equation}
\label{IV6}
D_{\alpha\omega}(A)=0.
\end{equation}
\end{proposition}

Now, we study the holomorphic last multipliers for the holomorphic Poisson bivector fields on a holomorphic Poisson manifold $(M,P,\omega)$ endowed with a holomorphic volume form. Let $f\in\mathcal{O}(M)$, $Z_f$ its holomorphic Hamiltonian vector field and $Z_\omega$ the holomorphic modular vector field associated with $(M,P,\omega)$. Taking into account that $Z_\omega(f)={\rm div}_\omega(Z_f)=D_\omega(Z_f)$ and $Z_f=\imath_{\partial f}P$ we obtain
\begin{equation}
\label{IV7}
Z_\omega(f)=(D_\omega(P))(f).
\end{equation}
According to \eqref{IV3}, $\alpha\in\mathcal{O}(M)$ is a holomorphic last multiplier for the holomorphic Poisson bivector field $P\in\mathcal{V}^2_\mathcal{O}(M)$ if
\begin{equation}
\label{IV8}
D_\omega(\alpha P)=0,\,\,{\rm or\,\,equivalently}\,\,D_{\alpha\omega}(P)=0.
\end{equation}
From \eqref{IV7} we have
\begin{proposition}
A function $\alpha\in\mathcal{O}(M)$ is a holomorphic last multiplier for the holomorphic Poisson bivector field $P$ if and only if
\begin{equation}
\label{IV9}
Z_{\alpha\omega}=0.
\end{equation}
\end{proposition}
If $(z^1,\ldots,z^n)$ is a local coordinates system on $M$ such that $\omega=dz^1\wedge\ldots\wedge dz^n$ and $P^{ij}$ are the local components of the holomorphic Poisson bivector $P$ of $M$, then using \eqref{III1}, the condition from \eqref{IV9} says that $\alpha\in\mathcal{O}(M)$ is a holomorphic last multiplier for $P$ if and only if
\begin{equation}
\label{IV10}
\sum_{j=1}^n\frac{\partial(\alpha P^{ij})}{\partial z^j}=0\,,\,\,i=1,\ldots,n.
\end{equation}

Let us reconsider now the previous examples for the case of holomorphic Poisson bivector fields.
\begin{example}
\label{e4.0}
Let $M$ be an $2$-dimensional complex manifold with local complex coordinates $(z^1,z^2)$, the holomorphic volume form $\omega=dz^1\wedge dz^2$ and the holomorphic Poisson bivector field $P=f(z^1,z^2)(\partial/\partial z^1)\wedge(\partial/\partial z^2)\in\mathcal{V}^2_\mathcal{O}(M)$. If $\alpha\in\mathcal{O}(M)$ is a holomorphic last multiplier for $P$ then \eqref{IV10} reads as follows:
\begin{displaymath}
\frac{\partial(\alpha f)}{\partial z^1}=\frac{\partial(\alpha f)}{\partial z^2}=0
\end{displaymath}
with the obvious solution $\alpha=A/f$, $A\in\mathbb{C}$ (it is assumed that $f$ is non-vanishing everywhere).
\end{example}
\begin{example}
Let us consider the holomorphic Lie-Poisson structure $P$ on $\mathbb{C}^n$ as in Example \ref{e3.1}. If $\alpha\in\mathcal{O}(\mathbb{C}^n)$ is a holomorphic last multiplier for $P$ then
\begin{equation}
\label{IV11}
\sum_{j=1}^nc^{ij}_k\frac{\partial(\alpha z^k)}{\partial z^j}=0\,,\,\,i=1,\ldots,n.
\end{equation}
If we consider the particular case $n=2$ with $c^{11}_1=c^{11}_2=c^{22}_1=c^{22}_2=0$ and $c^{12}_1,c^{12}_2\in\mathbb{C}$, then the general solution of \eqref{IV11} is $\alpha=A/(c^{12}_1z^1+c^{12}_2z^2)$, which also follows from Example \ref{e4.0}.
\end{example}
\begin{example}
On $\mathbb{C}^3$ take global complex coordinates $(z^1,z^2,z^3)$, $\omega=dz^1\wedge dz^2\wedge dz^3$ and any constant Poisson structure $P$ on $\mathbb{C}^3$ with local components $P^{ij}\in\mathbb{C}$, where $P^{ij}+P^{ji}=0$. Then, if $\alpha\in\mathcal{O}(\mathbb{C}^3)$ is a holomorphic last multiplier for the holomorphic Poisson bivector field $P$, the system \eqref{IV10} reads as follows:
\begin{equation}
\label{IV16}
\left\{
\begin{array}{ll}
\,\,\,\,\,\,\,\,\,\,\,\,\,\,\,\,\,\,\,\,\,\,\,\,P^{12}\frac{\partial\alpha}{\partial z^2}+P^{13}\frac{\partial\alpha}{\partial z^3}=0 &  \\
 &  \\
P^{21}\frac{\partial\alpha}{\partial z^1}\,\,\,\,\,\,\,\,\,\,\,\,\,\,\,\,\,\,\,\,\,\,\,\,\,+P^{23}\frac{\partial\alpha}{\partial z^3}=0 & \\
& \\
P^{31}\frac{\partial\alpha}{\partial z^1}+P^{32}\frac{\partial\alpha}{\partial z^2}\,\,\,\,\,\,\,\,\,\,\,\,\,\,\,\,\,\,\,\,\,\,\,\,\,=0 &
\end{array}
\right.
\end{equation}
with the general solution $\alpha=\phi(P^{23}z^1+P^{31}z^2+P^{12}z^3)$, where $\phi\in C^1(\mathbb{C})$.
\end{example}
\begin{example}
Let us consider the holomorphic Lie-Poisson bivector field $P$ on the dual $\mathfrak{sl}^*(2,\mathbb{C})$ of the Lie algebra of $\mathfrak{sl}(2,\mathbb{C})$, with holomorphic volume form  $\omega=dz^1\wedge dz^2\wedge dz^3$, that is, see \cite{C-F-I-U}
\begin{displaymath}
P=2z^2\frac{\partial}{\partial z^1}\wedge \frac{\partial}{\partial z^2}-2z^3\frac{\partial}{\partial z^1}\wedge \frac{\partial}{\partial z^3}+z^1\frac{\partial}{\partial z^2}\wedge\frac{\partial}{\partial z^3}.
\end{displaymath}
If $\alpha$ is a holomorphic last multiplier for $P$, then the system \eqref{IV10} reads as follows:
\begin{equation}
\label{IV14-1}
\left\{
\begin{array}{ll}
\,\,\,\,\,\,\,\,\,\,\,\,\,\,\,\,\,\,\,\,\,\,\,\,\,\,\,2z^2\frac{\partial\alpha}{\partial z^2}-2z^3\frac{\partial\alpha}{\partial z^3}=0 &  \\
 &  \\
-2z^2\frac{\partial\alpha}{\partial z^1}\,\,\,\,\,\,\,\,\,\,\,\,\,\,\,\,\,\,\,\,\,\,\,\,\,+z^1\frac{\partial\alpha}{\partial z^3}=0 & \\
& \\
\,\,\,\,\,2z^3\frac{\partial\alpha}{\partial z^1}-z^1\frac{\partial\alpha}{\partial z^2}\,\,\,\,\,\,\,\,\,\,\,\,\,\,\,\,\,\,\,\,\,\,\,\,\,=0 &
\end{array}
\right.
\end{equation}
with the general solution $\alpha=\varphi((z^1)^2+4z^2z^3)$, where $\varphi\in C^1(\mathbb{C})$.
\end{example}
\begin{example}
Let $\mathbb{C}^3$, global complex coordinates $(z^1,z^2,z^3)$, $\omega=dz^1\wedge dz^2\wedge dz^3$ and the holomorphic Poisson bivector field $P$ on $\mathbb{C}^3$ defined by the Poisson bracket from \eqref{III6} in Example \ref{e3.2}. Then, if $\alpha\in\mathcal{O}(\mathbb{C}^3)$ is a holomorphic last multiplier for the holomorphic Poisson bivector field $P$, the system \eqref{IV10} reads as follows:
\begin{equation}
\label{IV14}
\left\{
\begin{array}{ll}
\,\,\,\,\,\,\,\,\,\,\,\,\,\,\,\,\,\,\,\,\,\,\,\,\,\,\,\,\,\,\,\,\,\,\,\,\,\,\,\,\,\,\,\,\,\,\,\,\,(2z^3-z^1z^2)\frac{\partial\alpha}{\partial z^2}+(z^1z^3-2z^2)\frac{\partial\alpha}{\partial z^3}=0 &  \\
 &  \\
(z^1z^2-2z^3)\frac{\partial\alpha}{\partial z^1}\,\,\,\,\,\,\,\,\,\,\,\,\,\,\,\,\,\,\,\,\,\,\,\,\,\,\,\,\,\,\,\,\,\,\,\,\,\,\,\,\,\,\,\,\,\,\,\,\,+(2z^1-z^2z^3)\frac{\partial\alpha}{\partial z^3}=0 & \\
& \\
(2z^2-z^1z^3)\frac{\partial\alpha}{\partial z^1}+(z^2z^3-2z^1)\frac{\partial\alpha}{\partial z^2}\,\,\,\,\,\,\,\,\,\,\,\,\,\,\,\,\,\,\,\,\,\,\,\,\,\,\,\,\,\,\,\,\,\,\,\,\,\,\,\,\,\,\,\,\,\,\,\,\,=0 &
\end{array}
\right..
\end{equation}
Multiplying the first equation with $z^1$, the second equation with $z^2$ and the third equation with $z^3$, and suming, we get
\begin{equation}
\label{IV15}
\left(z^1(z^2)^2-z^1(z^3)^2\right)\frac{\partial\alpha}{\partial z^1}+\left(z^2(z^3)^2-z^2(z^1)^2\right)\frac{\partial\alpha}{\partial z^2}+\left(z^3(z^1)^2-z^3(z^2)^2\right)\frac{\partial\alpha}{\partial z^3}=0.
\end{equation}
The general solution of \eqref{IV15} is $\alpha=\phi\left((z^1)^2+(z^2)^2+(z^3)^2,z^1z^2z^3\right)$ with $\phi\in C^1(\mathbb{C}^2)$, and replacing in the first equation of \eqref{IV14} we obtain the general solution of \eqref{IV14} in the form
\begin{displaymath}
\alpha=\varphi\left((z^1)^2+(z^2)^2+(z^3)^2-z^1z^2z^3\right),
\end{displaymath}
where $\varphi\in C^1(\mathbb{C})$.
\end{example}

\noindent
Mircea Crasmareanu \\
Faculty of Mathematics, University "Al. I. Cuza" \newline
Address: Ia\c si, 700506, Bd. Carol I, no. 11, Rom\^ania \newline
email: \textit{mcrasm@uaic.ro}

\medskip

\noindent
Cristian Ida\\
Department of Mathematics and Computer Science, University Transilvania of Bra\c{s}ov\\
Address: Bra\c{s}ov 500091, Str. Iuliu Maniu 50, Rom\^{a}nia\\
email: \textit{cristian.ida@unitbv.ro}
\medskip

\noindent
Paul Popescu\\
Department of Applied Mathematics,  University of Craiova\\
Address: Craiova, 200585,  Str. Al. Cuza, No. 13,  Rom\^{a}nia\\
 email:\textit{paul$_{-}$p$_{-}$popescu@yahoo.com}

\end{document}